\documentclass[letter,12pt]{article}
\usepackage{amscd,amsmath,amssymb,amsxtra,amsthm,xcolor,array,enumerate
}
\usepackage{graphicx}
\usepackage{pgfplots,tikz-cd,tikz}
\usepackage{rab}

\title{Choreography of divisors\\ on algebraic curves}
\date{}
\author{Oleg Viro}
\begin{document}
\maketitle

\section{Introduction}\label{s1}
\subsection{The initial question}\label{s1.1} 
Let $X$ be a non-singular real algebraic plane projective curve and let  
another real algebraic plane projective curve $Y$ cut on $X$ real points
$P_1$, \dots, $P_n$.

{\em How can $P_1$, \dots, $P_n$
move along $X$, when $Y$ changes continuously?}

We assume that $Y$ remains transversal to $X$ and the 
degree of $Y$ does not change. 
Then the number of real intersection 
points does not change.  In particular, the points do not collide 
with each other and move continuously along $X$.
It looks like a motion of an ensemble of dancers on a curve. 

If $X$ is a rational curve, then the answer is purely topological. Any 
continuous isotopy of the intersection can be obtained in this way. 
For curves of positive genus, there are obstructions. They appear both 
infinitesimally (see section \ref{s1.3}),
and globally (see  sections \ref{s1.5} and \ref{s4}).

\subsection{Generalizations to linear equivalence classes of
divisors}\label{s1.2}
Intersections of a plane projective curve $X$ with other plane projective 
curves $Y$ of a fixed degree form a single class of linear equivalent 
divisors on $X$. This is  quite a special class. 

The question that was posed in Section \ref{s1.1} makes sense for 
any class of linear equivalent divisors. Let $X$ be a non-singular real
algebraic projective curve, and $D=P_1+\dots+P_n$ be a simple real divisor 
on $X$.

{\em How can $D$ move on $X$ continuously in its linear equivalence class?}
We assume that $D$ remains simple and real during the movement.
  
%

 %

\subsection{The velocities of points}\label{s1.3}
Linear equivalence classes of divisors are described by the Abel-Jacobi
theorem. It provides also a complete  
{\em infinitesimal answer\/} to the question on a movement: 

If the movement is smooth, then each point of a moving divisor
has a velocity vector. It follows from Abel-Jacobi Theorem that 

{\it a moving divisor does not leave its linear equivalence class 
if and only if, for any holomorphic differential form on the curve,  
the sum of values of the form on the velocity
vectors over all the points of the divisor vanishes.} 

Global restrictions on motions of a simple real divisor, which we are going
to study, look quite differently. 

\subsection{Topological choreography}\label{s1.4} If the 
algebraic restrictions on a moving configuration of points are dismissed 
(i.e., the points are not required to be cut by other curve or form a 
divisor linear equivalent to the same divisor), then a complete 
description of possible motions becomes easy. 

Let $T$ be a closed 1-manifold.  Connected components of $T$ are 
homeomorphic to the circle. Denote them by $T_1$, \dots, $T_r$.  
Consider the set $\UConf_d(T)$ of unordered 
configurations $\{P_1,\dots,P_d\}$ of $d$ pairwise distinct points in $T$. 
The set $\UConf_d(T)$ has a natural topology.
A connected component of  $\UConf_d(T)$ is determined by a distribution of 
points $P_1$, \dots, $P_d$ in components $T_1$, \dots, $T_r$ of $T$.

Consider a motion of  $\{P_1,\dots,P_d\}\in\UConf_d(T) $ which is closed in
the sense that the final configuration of points coincides with the initial 
one (although the points may shuffle).
In other words, consider a loop $L$ in  $\UConf_d(T)$. 

The paths which are traced by individual points of a configuration 
constitute a 1-cycle on $T$. Its 
homology class belongs to $H_1(T)$. We will call this homology class 
the {\em tracing\/} of $L$ and denote it by $\tr(L)$. It does not change
under a homotopy of $L$.    

The homology  $H_1(T)$ is a free abelian group generated by
the fundamental classes $[T_1]$, \dots, $[T_r]$ of the components. 
Hence $\tr(L)=\sum_{k=1}^rc_k[T_k]$, for some $c_k\in\Z$.


{\bf Homotopy classification of loops in $\UConf_d(T)$.}
{\it  A loop $L$ \ in \  $\UConf_d(T)$ which begins at $\{P_1,\dots,P_d\}$ 
is defined up to homotopy by 
the distribution of points \ $P_1$, \dots, $P_d$ \ in \ $T_1$, \dots,
$T_r$
and  tracing $\tr(L)=\sum_{k=1}^rc_k[T_k]\in H_1(T)$.   
The only restriction on $\tr(L)$ is that if \ 
$T_k$ contains no points from $\{P_1,\dots,P_d\}$, then $c_k=0$.}

The proof is straightforward.\qed 

\subsection{Choreography of divisors}\label{s1.5}
For loops in the space of real simple divisors linear equivalent
to each other, the tracing satisfies many extra restrictions, which 
depend of the curve and divisor.  We postpone the detailed statements till section \ref{s4}. They are 
the main results of this paper.

\subsection{Acknowledgements}\label{s1.6}
The first interesting special
case  emerges, when a plane projective M-cubic is cut by another cubic. 
It was considered by Ay\c{s}eg\"ul \"Ozt\"urkalan
\cite{Ay}.
She relied on a completely different technique, elliptic addition on an 
M-cubic. As is shown below, most of her results do not require
elliptic addition and hold true in much more general situation. However, it
was her results that have inspired and motivated the present work.

I am grateful to Grisha Mikhalkin, John Milnor, Kristin Shaw, Dennis
Sullivan and Julia Viro for valuable remarks and suggestions.

\section{Preliminary notions}\label{s2}

\subsection{A real curve}\label{s2.1} 
In this paper almost everything 
happens on an {\em algebraic curve\/} $X$. 
We will focus on the situations 
when this ambient curve is {\em non-singular, irreducible and
projective\/}.  In this section, we assume also that the ambient curve is 
{\em real.\/}     

We do not assume that $X$ is planar or embedded in any specific standard
algebraic variety. Although the curve is assumed to be projective, the 
ambient projective space will not play any r\^ole.  
We would rather think of $X$ as an abstract complete curve. 
The set of its real points is denoted by $\R X$, the set of its complex
points, by $\C X$. Projectivity of $X$ implies that both $\C X$ and 
$\R X$ are compact. 

From analytical perspective, $\C X$ is a closed 
non-singular Riemann surface with an anti-holomorphic involution $\conj$, 
and $\R X$ is the  fixed point set of $\conj$.  


\subsection{Complex topological characteristics of a real curve}\label{s2.2}
A curve $X$ is said to be of {\em type I\/} if $\R X$ bounds in $\C X$ (i.e., 
$\R X$ realizes $0\in H_2(\C X;\Z/{2\Z})$). Otherwise $X$ is said to be of
{\em type II\/}. 

If $X$ is of type I, then $\C X\sminus \R X$ consists of two connected 
components homeomorphic to each other and interchanged by $\conj$. The
orientation of a connected component of $\C X\sminus \R X$ defines an
orientation on $\R X$ (as on the boundary of its closure). 
The other connected component of $\C X\sminus\R X$
defines the opposite orientation on $\R X$. These orientations are called
{\em complex orientations\/} because they come from the orientation of 
$\C X$, which is determined by the complex structure of $\C X$. 

A connected component $X_i$ of $\R X$ is homeomorphic to circle.  
If $X$ is of type I, then a complex orientation of $\R X$ defines 
orientations of $X_i$. With these orientations, the components define 
the basis $[X_1],\dots,[X_r]$ of $H_1(\R X)$. The sum $[X_1]+\dots+[X_r]$
of these generators is called a {\em complex orientation class\/} of $X$.
It defines and is defined by a complex orientation. The
inclusion homomorphism $H_1(\R X)\to H_1(\C X)$ has an infinite cyclic
kernel generated by a complex orientation class. 

If $X$ is of type II, then $\C X\sminus\R X$ is connected and the inclusion
homomorphism $H_1(\R X)\to H_1(\C X)$ is injective.

The types were introduced by Felix Klein \cite{Kl} in 1876. The complex
orientations of type I curves were introduced by V.~A.~Rokhlin \cite{R1}, 
\cite{R2}.

\subsection{Real divisors}\label{s2.3}
Let $D$ be a formal linear combination $n_1P_1+\dots+n_lP_l$ of points 
$P_1,\dots,P_l\in \C X$ with {\em positive integer\/} coefficients
$n_1,\dots,n_l$. In other words, let $D$ be an effective divisor on $X$.
Recall that the natural number $n_1+\dots+n_l$ is called the {\em degree\/}
of divisor $D=n_1P_1+\dots+n_lP_l$ and denoted by $\deg D$. 
The set $\{P_1,\dots,P_l\}$ is called
the {\em support\/} of $D$ and denoted by $\supp D$.

A divisor $D=n_1P_1+\dots+n_lP_l$ is said to be {\em real\/} if it is 
invariant under $\conj$. This means that each $P_i$ is either real (i.e.,
$P_i=\conj P_i$), or $D$ contains together with the summand $n_iP_i$ 
also the conjugate summand $n_jP_j=n_i\conj P_i$.  

A real effective divisor
$D$ admits a unique decomposition $D=D'+D''$ into an effective divisor 
$D'$ with $\supp D'\subset\R X$ 
and an effective divisor $D''$ with 
$\R X\cap\supp D''=\varnothing$. Then $D'$ is called the {\em real
part\/} of $D$ and denoted by $\R D$, and $D''$ is called the
{\em imaginary part\/} of $D$ and denoted by $\Im D$.
The degree of $\R D$ is called {\em real
degree\/} of $D$. If the real degree of $D$ equals degree (i.e., if
$D=\R D$ and $\Im D=0$), then $D$ is said to be {\em purely real\/}.

\subsection{Divisors form a space}\label{s2.4}
 An effective divisor of degree $d$  
can be identified with a point of the {\em $d$-th symmetric power\/} 
$\Sym^d(\C X)$. Recall that the $d$th symmetric power $\Sym^dT$ of a 
topological space 
$T$ is the quotient space of the Cartesian power $T^d$ under the natural 
action of the symmetric group $\mathbb S_d$.

As is well-known, {\it the space $\Sym^d(\C X)$ is a manifold of 
dimension $2d$ 
and can be equipped with a complex structure
such that the natural projection $(\C X)^d\to\Sym^d(\C X)$ is
holomorphic.\/} 
This can be easily shown using 
a well-known bijection $\Sym^d\C\to\C^d$ which maps 
$\{x_1,\dots,x_d\}\in\Sym^d\C$ to $(a_1,\dots,a_d)\in\C^d$ such that
$(x-x_1)\dots(x-x_d)=x^d+a_1x^{d-1}+\dots+a_d.$ 

\subsection{Linear equivalence of divisors}\label{s2.5}
Remind that some divisors $D=\sum_{k=1}^ln_kP_k$ on $\C X$ can be 
presented as $f^*(v)$ for a regular (i.e., holomorphic) map  
$f:\C X\to\C P^1$ and $v\in\C P^1$. This means that 
$f^{-1}(v)=\{P_1,\dots,P_l\}$ and 
$n_k$ is the multiplicity of $P_k$ as a root of equation $f(x)=v$ 
for each $k$. 

Recall that effective divisors $D$ and $E$ are said to be 
{\em linear equivalent\/}
if there exists a regular map $f:\C X\to\C P^1$ such that
$f^*(v)+E=D+f^*(w)$  for some $v,w\in \C P^1$. Linear equivalence is indeed
an equivalence relation, see e.g., \cite{Sh}. 

Denote by $\Leq(D)$ the set of 
effective divisors on $X$ linearly equivalent to an effective 
divisor $D$. The set $\Leq(D)$ is a subvariety of $\Sym^{\deg D}\C X$ 
isomorphic to a finite dimensional complex projective space, see e.g., 
\cite{Sh}.

\subsection{Linear equivalence of real divisors}\label{s2.6}
If $D$ is a real effective divisor on $X$, 
then $\conj$ induces an anti-holomorphic involution on $\Leq(D)$ and turns 
it into a real algebraic variety. 

The subset of $\Leq(D)$ which is formed 
by non-simple divisors is a hypersurface, it is called the {\em discriminant
hypersurface\/} and denoted by $\GS \Leq(D)$. 

\subsection{Deformations of real divisors}\label{s2.7}
Let  $D$ be a simple real divisor on $X$. 
The connected component of $\R \Leq(D)\sminus\GS \Leq(D)$ which contains $D$ 
is denoted by $\Def(D)$.
It can be described as the space of {\em simple real\/} divisors on $X$ which are
linearly equivalent to $D$ and can be connected to $D$ by {\em a continuous 
family of simple real divisors.\/} A path in $\Def(D)$ gives rise to an 
isotopy 
of $\supp D$ in the process of which $\supp\R D$ moves in $\R X$, while
$\supp\Im D$ moves in 
$\C X\sminus \R X$.\medskip

\noindent{\itshape\bfseries Example.}
Let $X$ be a real plane curve, and $Y$ another real plane curve transversal 
to $X$. Let $D$ be a simple divisor on $X$ which is cut by $Y$ on $X$, that
is $\supp D=\C X\cap\C Y$.
Then for each $E\in \Def(D)$ the set $\supp E$ is a 
transversal intersection of $\C X$ with $\C Y'$, where $Y'$ is a plane curve 
of the same degree as $Y$ and $\C X\cap\C Y'$ is isotopic to $\C X\cap\C Y'$
in the space of finite subsets of $\C X$ invariant under $\conj$.

\section{Tracings of a real moving simple divisor}\label{s3}
\subsection{Real and imaginary tracing maps}\label{s3.1}
Let $D$ be a simple real divisor in $X$ with 
$\R D=\sum_{k=1}^nP_k$ and $\Im D=\sum_{k=1}^{m}(P_{n+k}+\conj P_{n+k})$. 

Let $L:[0,1]\to \Def(D)$ be a loop with $L(0)=L(1)=D$. 
Then there are uniquely defined  
paths $L_1,\dots,L_n:I\to\R X$ and $L_{n+1},\dots,L_{n+m}:I\to\C X\sminus\R
X$ such that $L(t)=\sum_{k=1}^{n}L_k(t)+\sum_{k=1}^{m}(L_{n+k}(t)+\conj
L_{n+k}(t))$ 
for any $t\in I$. 

Each of these paths can be considered as a singular 1-simplex. 
Notice that singular chains $L_1+\dots+L_n$ and
$(L_{n+1}+\conj\circ L_{n+1})+\dots+(L_{n+m}+\conj\circ L_{n+m})$ are 
cycles, because each of
the paths starts at a point where exactly one of the paths finishes. 

The cycle $L_1+\dots+L_n$ determines a homology class which is an 
element of $H_1(\R X)$, the cycle 
$(L_{n+1}+\conj\circ L_{n+1})+\dots+(L_{n+m}+\conj\circ L_{n+m})$
determines a homology class which is an element of $H_1(\C X\sminus\R X)$.
We will call these homology classes {\em real\/} and {\em imaginary
tracing\/} classes of $L$, respectively, and denote them 
by $\Re\tr(L)$ and $\Im\tr(L)$.

 It is easy to see that the tracing homology classes 
depend only on the homotopy class of $L$ (a homotopy between loops in
$\Def(D)$ gives rise to a homology between the tracing cycles exactly in
the same way as $L$ turns into tracing cycles  $\Re\tr(L)$ and $\Im\tr(L)$). 
The maps
\begin{align*} 
&\Re\tr:\pi_1(\Def(D))\to H_1(\R X),\\
&\Im\tr:\pi_1(\Def(D))\to H_1(\C X\sminus\R X) 
\end{align*}  
defined by this construction will be called 
{\em real\/} and {\em imaginary tracing maps\/}, respectively.  
One can easily check that they are homomorphisms.

\subsection{Encoding of real tracing}\label{s3.2}
Let $X_1$, \dots, $X_r$ be connected components of $\R X$. Each of them is
homeomorphic to circle. An orientation of $\R X$ gives generators $[X_i]\in
H_1(X_i)$, which identify $H_1(X_i)$ with $\Z$.  The group $H_1(\R X)=
\oplus_{i=1}^rH_1(X_i)$ is identified with $\Z^r$. 

Thus, an orientation of $\R X$ encodes the real tracing class
$\Re\tr(L)$ of a loop $L:I\to \Def(D)$ by a sequence of integers 
$(c_1(L),\dots,c_r(L))$ such that  $\Re\tr(L)=c_1(L)[X_1]+\dots+c_r(L)[X_r]$.

Recall that if $X$ is a curve of type I, then a complex orientation 
provides a natural set of generators of $H_1(X_i)$, see Section \ref{s2.2}. 

\subsection{Monodromy permutations}\label{s3.3}
The paths $L_1,\dots,L_n:I\to\R X$ defined by a loop $L:I\to \Def(D)$ 
form an isotopy of the set $A=\supp\R D$ in $\R X$. The
isotopy defines a monodromy map  $\mu_L:A\to A$ that maps 
$P_k=L_k(0)$ to $L_k(1)$ for $k=1,\dots,n$. 

Of course, points $P_k$ and $\mu_L(P_k)$ belong to the same
connected component of $\R X$. Denote $A\cap X_i$ by $A_i$. 
Thus, $\mu_L$ maps $A_i$ to itself.

An orientation of $X_i$ induces a cyclic order on $A_i$.
An isotopy of a finite subset $A_i$ of $X_i$ preserves the cyclic order 
on $A_i$.  Therefore the monodromy map  $\mu_L:A\to A$ restricted to $A_i$ 
is a cyclic permutation. 

The group of permutations preserving a cyclic order on a finite set
has a canonical generator which sends each
element to the next one. Let us denote this canonical generator by $\nu$. 

The monodromy map $\mu_L:A\to A$ is completely determined by the
tracing class: 
$$\mu_L|_{A_i}=\nu^{c_i(L)},$$
where $c_i(L)$ are the integers which encode
$\Re\tr(L)=\sum_{k=1}^rc_k(L)[X_k]$, see Section \ref{s3.2}. 

\subsection{On a complex curve}\label{s3.4} Let $X$ be a non-singular
{\em complex\/} curve and $D$ be a simple divisor on $X$. As above, we
denote by $\Leq(D)$ the variety of divisors on $X$ which are linear
equivalent to $D$, and by $\GS\Leq(D)$ the discriminant hypersurface 
of $\Leq(D)$ which consists of non-simple divisors. 

The complement $\C\Leq(D)\sminus\C\GS\Leq(D)$ is a counter-part of 
$\Def(D)$: it consists of simple divisors which are linearly equivalent to
$D$ and can be connected to $D$ by a continuous family of simple divisors.
The only difference is that divisors are not required to be real, cf.
section \ref{s2.7}. 

Let $L:[0,1]\to \C\Leq(D)\sminus\C\GS\Leq(D)$ be a loop with $L(0)=L(1)=D$. 
Then there are uniquely defined  
paths $L_1,\dots,L_d:I\to\R C$ such that $L(t)=\sum_{k=1}^{d}L_k(t)$ 
for any $t\in I$. 

Each of these paths can be considered as a singular 1-simplex. 
Notice that singular chains $L_1+\dots+L_d$  is a 
cycle, because each of
the paths starts at a point where exactly one of the paths finishes. 

The cycle $L_1+\dots+L_n$ determines a homology class which is an 
element of $H_1(\C X)$. We will call this homology class the {\em 
tracing\/} class of $L$ and denote it by $\tr(L)$.

The tracing homology class $\tr(L)$ 
depends only on the homotopy class of $L$. A homotopy between loops in
$\C\Leq(D)\sminus\C\GS\Leq(D)$ gives rise to a homology between the tracing 
cycles exactly in the same way as $L$ turns into the tracing cycle
$\tr(L)$. 
The map 
$$\tr:\pi_1(\C\Leq(D)\sminus\C\GS\Leq(D))\to H_1(\C X)$$
defined by this construction will be called 
{\em tracing map\/}.  
One can easily check that it is  a homomorphism.

\subsection{Relation among tracing maps}\label{s3.5} If $X$ is a
real non-singular curve, $D$ is a simple real divisor on $X$, then the
following diagram is obviously commutative
\begin{equation}\label{tr-diagram} 
\begin{tikzcd} 
\pi_1(\Def(D))\arrow[r,"\Re\tr\oplus\Im\tr"]\arrow[d,"\inc_*"]
& H_1(\R X)\oplus H_1(\C X\sminus\R X)\arrow[d,"\inc_*+\inc_*"]\\
\pi_1(\C\Leq(D)\sminus\C\GS\Leq(D))\arrow[r,"\tr"]&H_1(\C X)
\end{tikzcd}
\end{equation}  
That is, $\tr\circ\inc_*=\inc_*\circ\Re\tr+\inc_*\circ\Im\tr$.

\section{Tracing theorems}\label{s4}
\subsection{Real tracing of a purely real divisor}\label{s4.1}

\begin{Th}\label{Th1}
Let $X$ be a non-singular projective real algebraic curve and let $D$ be a
simple purely real divisor on $X$.\\
\hspace*{21pt}{\bf(a) } If $X$ is of type I, then the image of real 
tracing map is contained in the cyclic subgroup generated by a complex
orientation class.\\
\hspace*{20pt}{\bf(b) } If $X$ is of type I and $\R X$ 
has a connected component disjoint from $\supp D$, then the real tracing map
$\Re\tr:\pi_1(\Def(D))\to H_1(\R X)$ is trivial.\\
\hspace*{22pt}{\bf(c) } If $X$ is of type II, then
the real tracing map
$\Re\tr:\pi_1(\Def(D))\to H_1(\R X)$ is trivial.
\end{Th}

\begin{rem}\label{rem4.1.1} 
The parts (b) and (c) of Theorem \ref{Th1} describe $\Re\tr$ completely,
while the part (a) gives only an upper bound for the image
of $\Re\tr$. In particular, the part (a) implies the following statement.
\end{rem}

\begin{cor}\label{cor} 
Let $X$ be a non-singular projective real algebraic curve of type I, let
$\R X$ be equipped with a complex orientation and consist of connected
components  $X_1,\dots,X_r$. 
Let $D$ be a simple purely real divisor on $X$, let $A=\supp D$ and 
$A_i=X_i\cap A$ be cyclically ordered via the complex
orientation of $X_i$. 
If $L:I\to \Def(D)$ is a loop, then:

{\bf(a)} $\Re\tr(L)=C_L[\R X]$ for some integer $C_L$ and

{\bf(b)} the induced monodromy $\mu_L:A\to A$ restricted to $A_i$ is $\nu^{C_L}$. 
\end{cor}

\subsection{Real tracing of real, but not necessarily purely real divisors}\label{s4.2}
The assumptions of the following theorem are weaker than the assumptions
of Theorem \ref{Th1}: the divisor is not assumed to be purely real.
 
\begin{Th}\label{Th2} Let $X$ be a non-singular projective real algebraic
curve of type I and $D$ be a real divisor on $X$. 
 
{\bf(a) } The image of real tracing map is 
contained in the cyclic subgroup generated by a complex
orientation class and classes divisible by two.

{\bf(b) }  If $\R X$ has a connected
component disjoint from $\supp D$,  then the image of real tracing map
$\Re\tr:\pi_1(\Def(D))\to H_1(\R X)$ consists of classes divisible by two.
\end{Th}

Naturally, the conclusions in Theorem \ref{Th2} are also weaker than their
counter-parts in Theorem \ref{Th1}. 
Namely, in the parts (a) and (b) 
the subgroup, where the image of $\Re\tr$ is allowed to be contained,
increases by the set of all even elements. In other words, the conclusions
stay the same, but in $\Z_{/2}$-homology instead of
$\Z$-homology. 

The part (c) of Theorem \ref{Th1} has no counter-part in
Theorem \ref{Th2}: Theorem \ref{Th2} does not contain any restriction 
on the real tracing map, if the curve is of type II and the divisor is 
not purely real.
 As we will see later, for a curve of type II the image
of $\Re\tr$ can be the whole $H_1(\R X)$.

\subsection{Tracing of a simple divisor on a complex curve}\label{s4.3} 
Theorems \ref{Th1} and \ref{Th2} are deduced below in section \ref{s6}
from the well-known facts about relation of homology classes realized 
in $H_1(\C X)$ by components of $\R X$ (see Section \ref{s2.2}), the 
relation among tracing homomorphisms described in section \ref{s3.5}, and
the following Theorem \ref{Th3}.  
 
\begin{Th}\label{Th3}
Let $X$ be a complex non-singular projective curve and $D$ be a simple
divisor on $X$. Then the tracing map  
$\tr:\pi_1(\C\Leq(D)\sminus\C\GS\Leq(D))\to H_1(\C X)$
is trivial.
\end{Th}

\section{Tracing for non-simple divisors}\label{s5}

\subsection{Why and how to extend tracing to non-simple divisors}\label{s5.0}
Tracing homomorphisms appeared in a number of situations above, see sections
\ref{s1.4}, \ref{s3.1} and \ref{s3.4}. They have been defined for 
spaces of {\em simple divisors}. However, they can 
be defined for {\em symmetric powers\/}, like the space $\C\Leq(D)$, 
which consists of divisors that are {\em not necessarily simple}.
 
The restriction to simple divisors makes the constructions easier. 
On the other hand, it suffices for stating the main results of this paper.

However, for the proof of Theorem \ref{Th3} we have to factor the tracing map
$\tr:\pi_1(\C\Leq(D)\sminus\C\GS\Leq(D))\to H_1(\C X)$ 
through $\pi_1(\C\Leq(D))$. This is done below by defining tracing
$\pi_1(\Sym^d(\C X))\to H_1(\C X)$, that is by  
allowing non-simple divisors. 

As a matter of fact, tracing homomorphism $\pi_1(\Sym^d(T))\to H_1(T)$
can be defined for any ``non-pathological'' topological space $T$ (say, a 
cw-complex). Here is a sketch for the construction:

As in the constructions above, it starts with taking a loop 
$L:I\to\Sym^d(T)$ which represents an element of $\pi_1(\Sym^d(T))$ and
splitting it into $d$ paths $I\to T$. They are not uniquely
defined by $L$, since the natural projection $p:\Sym^d(T)\to T$ is not a
covering. In order to overcome this difficulty, let us make $L$ nice with 
respect to the natural stratification of $\Sym^d(T)$, so that one can 
decompose $I$ into cells such that, on the interior of each of the 1-cells, 
$L$ would go in one of the strata. 
Then on the restrictions of $L$ to each of the 1-cells can be split 
to paths in $T$. The paths considered as singular 1-simplexes   
would form a 1-cycle, and we take the homology class of this 1-cycle. 

Then we would need to prove that the homology class does not depend on the
choices made in this construction and that homotopy of $L$ does not change 
the homology class. The former is obvious. The latter can be made 
by choosing the homotopy of $L$ nice  with respect to the natural
stratification of $\Sym^d(T)$ and subdividing it to nice pieces, and
proceeding as above, when we built the 1-cycle. 

This is how the author convinced himself that a required tracing exists,
but writing down all the details seemed be quite cumbersome and encouraged
to look for simplifications. 

We need this construction only for $T=\C X$. 
Since $\Sym^d(\C X)$ is a manifold and the discriminant 
$\GS$  has a real codimension 2, the loop $L$ which represents an
element of $\pi_1(\Sym^d(\C X))$ can be chosen in the complement of $\GS$. 
Then it splits into paths $L_1,\dots,L_d\to\C X$ uniquely. 

A homotopy between loops in $\Sym^d(T)$ may hit $\GS$. 
However, it can be chosen meeting $\GS$ in a generic point 
transversally. We will do this in section \ref{s5.6} below.
As a by-product,  a more 
conceptual view on relation between maps of arbitrary space to $\Sym^d(T)$ 
and $T$ will be developed in  sections \ref{s5.4} and \ref{s5.5}. 

There exists yet another construction for tracing
$\pi_1(\Sym^d(T))\to H_1(T)$ based on the Dold-Thom 
isomorphism. It is presented in section \ref{s5.7} for the sake of
completeness. The construction is
elegant and conceptual, but difficult to relate to the constructions that 
we use in configuration spaces. 
  
\subsection{Configuration spaces}\label{s5.1}
Recall that spaces of effective divisors on a curve $X$ were introduced
above as symmetric powers $\Sym^d(\C X)$ of $\C X$ equipped with additional 
structures, and spaces $\Leq(D)$ and $\Def(D)$ were introduced as subspaces 
of $\Sym^d(\C X)$, see sections \ref{s2.4}, \ref{s2.6} and \ref{s2.7}. 

Let $T$ be a topological space.
Points of $T^d$ are ordered $d$-tuples $(x_1,\dots,x_d)$ of points of $T$.
Points of $\Sym^dT$ are unordered $d$-tuples $\{x_1,\dots,x_d\}$ of points 
of $T$. A point 
can appear several times, so the number of pairwise different points in 
a $d$-tuple may be less than $d$. A $d$-tuple of pairwise different points 
of $T$ is called a {\em configuration\/} of $d$ points. 

The subset of $T^d$ which consists of points $(x_1,\dots,x_d)$, 
where $x_i=x_j$ for some $i\ne j$, is called the {\em big diagonal\/} of
$T^d$ and is denoted by $\GD$. Denote by $\GS$ the image of $\GD$ under the natural
projection $p:T^d\to\Sym^dT$, that is $\GS=p(\GD)$.

The space $T^d\sminus\GD$ is the {\em $d$-th ordered configuration space\/}
 of $T$ and denoted $\Conf_d(T)$. The space $\Sym^dT\sminus\GS$ is called
the {\em $d$-th unordered configuration space\/} and denoted by
$\UConf_d(T)$.
The map $\Conf_d(T)\to\UConf_d(T)$ defined as a submap of $p:T^d\to\Sym^dT$ 
is a regular $\mathbb S_d$-covering.

\subsection{Tracing in configuration spaces}\label{s5.2}
In this section we construct a homomorphism 
$\Tr:\pi_1(\UConf_d(T))\to H_1(T)$ for any positive integer $d$
and a topological space $T$. We follow almost literally the constructions 
of section \ref{s3.4}.

Shortly speaking, $\Tr$ assigns to a homotopy class 
of a loop in the configuration space $\UConf_d(T)$ the homology class 
of the sum of paths 
which are traced by all the points of the configuration while the
configuration traces the loop. Here are more 
details.

Let  $L:I\to\UConf_d(T)$  be a loop. By the Lifting Path Theorem, 
there exists a path   $\widetilde L:I \to\Conf_d(T)$ covering $L$. 
It can be presented as $\widetilde L:t\mapsto (L_1(t),\dots,L_d(t))$, 
where $L_k:I\to T$
with $k=1,\dots,d$ are paths in $T$. Since $L$ is a loop, i.e.,
$L(0)=L(1)$,  we have
$\{L_1(0),\dots,L_d(0)\}=\{L_1(1),\dots,L_d(1)\}$.
It follows that the singular chain $L_1+\dots+L_d\in C_1(T)$ is a cycle.
Denote by $\Tr(L)$ the homology class of this cycle in $H_1(T)$.

In the construction of $\Tr(L)$, we have chosen a covering path 
$\widetilde L$ of $L$. A covering path is determined by a choice of its
initial point. The initial point of $L$ is a $d$-element set 
$\{x_1,\dots,x_d\}\subset T$. An initial point of the covering path is a
point of $T^d$ obtained by any ordering of this set. 
The ordering of the initial points is nothing but an ordering of paths
$L_k$. It does not affect the homology class of $L_1+\dots+L_d$. Thus, 
$\Tr(L)$ does not depend on the choice.

For $\Gl\in\pi_1(\UConf_d(T))$, define $\Tr(\Gl)$ to be $\Tr(L)$
for a loop $L$ which belongs to $\Gl$. It does not depend on the choice of
$L$. Indeed, if $L$ and $L'$ are homotopic loops in $\UConf_d(T)$, 
then homotopy between them  can be lifted to $ \Conf_d(T)$ by
the Lifting Homotopy Theorem and gives rise to a homology between the
cycles $\Tr(L)$ and $\Tr(L')$ in $T$.
\hfill$\square$

\subsection{Relation between the tracing maps}\label{s5.3}
If $X$ is a non-singular complex projective curve and $T=\C X$ and $D$ is a
simple real divisor of degree $d$ on $X$, then $\C\Leq(D)\sminus\C\GS\Leq(D)$ is a subspace of
$\UConf_d(T)$. The tracing map 
$\tr:\pi_1(\C\Leq(D)\sminus\C\GS\Leq(D))\to H_1(\C X)$  
defined in section \ref{s3.4} is factored through $\Tr$ via the
following commutative diagram.

\begin{equation}\label{eq1}
\begin{tikzcd}
\pi_1(\C\Leq(D)\sminus\C\GS\Leq(D))\arrow[rd,"\tr"]\arrow[r,"\inc_*"]&\pi_1(\UConf_d(\C X))\arrow[d,"\Tr"]\\ 
&H_1(\C X)
\end{tikzcd}
\end{equation}

\subsection{Spaces of pointed unordered $d$-tuples}\label{s5.4} An
unordered $d$-tuple of points, in which one of the points is 
distinguished (we will say also {\em pointed\/}), is called {\em pointed.\/} 
The space of pointed unordered
$d$-tuples of points of a space $T$ can be identified with 
$(\Sym^{d-1}T)\times T$. Indeed, 
$(\{x_1,\dots,x_{d-1}\},x)\in(\Sym^{d-1}T)\times T$ can be identified
with the unordered $d$-tuples $\{x_1,\dots,x_{d-1},x\}\in \Sym^dT$ in 
which $x$ is distinguished. 

 Like $\Sym^dT$, the space $(\Sym^{d-1}T)\times T$
can be obtained as a quotient space of $T^d$. Indeed,
$(\Sym^{d-1}T)\times T$ is the orbit space of the action of $\mathbb S_{d-1}$ 
in $T^d$ by permuting the first $d-1$ components and keeping the last
component fixed. 
Each orbit of this  action is contained in an orbit of the natural action 
of $\mathbb S_d$ in $T^d$. This gives a map $q:(\Sym^{d-1}T)\times T\to\Sym^dT$.
In terms of pointed $d$-tuples, $q$ forgets which element of $d$-tuple was
pointed. For an unordered $d$-tuple $\{x_1,\dots,x_d\}\in\Sym^dT$, 
the preimage under $q$ consists of the same $d$-tuple in which pointing is 
made in all the ways.

Thus the space of pointed unordered $d$-tuples of points of $T$ occupies 
an intermediate position between $T^d$ and $\Sym^dT$: there are natural 
surjections $T^d\to(\Sym^{d-1}T)\times T\to\Sym^dT$.

Over $\UConf_d(T)=\Sym^dT\sminus\GS$, the map $q$ defines a covering  
$$\left((\Sym^{d-1}T)\times T\right)\sminus q^{-1}\GS\to\Sym^dT\sminus\GS$$
The space $\left((\Sym^{d-1}T)\times T\right)\sminus q^{-1}\GS$ can be
interpreted as the space of {\em pointed\/} unordered $d$-element subsets 
of $T$.  Denote this space by $\PConf_d(T)$ and the covering 
$\PConf_d(T)\to\UConf_d(T)$ forgetting the pointing, by $r$.

\subsection{Generalizations of tracing map}\label{s5.5}
In this section 
for any topological space $Y$ and its continuous map $f:Y\to\UConf_d(T)$
we will construct a $d$-fold covering space $\widetilde Y$ of $Y$ and a 
continuous map $\widetilde f:\widetilde Y\to T$. 

The covering $\widetilde Y\to Y$ is the pullback $f^*r$ of the covering 
$r:\PConf_d(T)\to\UConf_d(T)$ 
by $f$. By the definition of pullback, the total space $\widetilde Y$ 
consists of pairs $(y,z)\in Y\times\PConf_d(T)$ such that $f(y)=r(z)$. 
 The map $\widetilde f$ maps $(y,z)\in\widetilde Y$ to the distinguished 
point of $z$. (Recall that $z\in\PConf_d(T)$ is a pointed unordered
configuration of points of $T$.) 

All these maps are gathered in the following
commutative diagram:
\begin{equation}\label{cd1} 
\begin{tikzcd}
\widetilde Y\arrow[d,"f^*r"]\arrow[r,"r^*f"]\ar[rrr,bend
left=30,"\widetilde f"]
&\PConf_d(T)\arrow[d,"r"]\arrow[r,hook]
&(\Sym^{d-1}T)\times T\arrow[d,"q"]\arrow[r,"pr_{\scriptscriptstyle
 T}"]&T\\
Y\arrow[r,"f"]&\UConf_d(T)\arrow[r,hook]&\Sym^dT& 
\end{tikzcd}
\end{equation}

The construction has several obvious nice properties:
\begin{enumerate}\setlength\itemsep{0pt} \setlength\topsep{0pt} 
\item If $Y$ is a manifold, then $\widetilde Y$ is a manifold.
\item An orientation of a manifold $Y$ defines an orientation of 
$\widetilde Y$.
\item If $f$ is homotopic to $g$, then $\widetilde f$ is homotopic to $\widetilde g$.
\item If $Y$ is an oriented manifold and $f:Y\to\UConf_d(T)$ is bordant to
zero, then $\widetilde f$ is bordant to zero.  
\end{enumerate}
These properties obviously imply the following theorem.

\begin{Th}\label{Th-tr}
For any topological space $T$ and any natural numbers $d$ and $k$, there is 
a natural homomorphism $\Omega_k(\UConf_d(T))\to \Omega_k(T)$ which maps 
a bordism class of a map $f:M\to\UConf_d(T) $ to a bordism class of 
$\widetilde f:\widetilde M\to T$.
\end{Th}

Since the 1-dimensional bordism group is naturally isomorphic the
1-dimensional homology group, Theorem \ref{Th-tr} provides a homomorphism 
$$H_1(\UConf_d(T))\to H_1(T)$$ It is easy to see that its composition 
$\pi_1(\UConf_d(T))\to H_1(\UConf_d(T))\to H_1(T)$ with
the Hurewicz homomorphism coincides with homomorphism $\Tr$ introduced in 
Section \ref{s5.2}.

\subsection{Extension of tracing}\label{s5.6}
\begin{Th}\label{Th5} 
If $X$ is a non-singular complex projective curve  and $d\ge2$ is an
integer, then there exists a homomorphism
$\TR:\pi_1(\Sym^d(\C X))\to H_1(\C X)$ such that the following diagram 
is commutative
\begin{equation}\label{eq2}
\begin{tikzcd} \pi_1(\UConf_d(\C
X))\arrow[r,"\inc_*"]\arrow[d,"\Tr"]& 
\pi_1(\Sym^d(\C X))\arrow[ld,"\TR"]\\H_1(\C X)&
\end{tikzcd}
\end{equation}
\end{Th}

\begin{proof}  Since $\GS$ is a complex hypersurface in
$\Sym^d(\C X)$, its real codimension is 2 and the inclusion homomorphism 
$$\inc_*: \pi_1(\UConf_d(\C X))(=\pi_1(\Sym^d(\C
X)\sminus\GS))\to\pi_1(\Sym^d(\C X))$$  
is surjective. Consider the composition 
$$
\pi_1(\Sym^d(\C X))\xrightarrow{\inc_*^{-1}}\pi_1(\UConf_d(\C X))
\xrightarrow{\Tr}H_1(\C X)
$$
A priori, this is a multivalued map, since $\inc_*$ is not injective.
In order to prove that this multivalued map is univalued, we have 
to verify that $\Tr(\Ker\inc_*)=0$. Down to loops, the latter means
that for any continuous map $f:S^1\to\Sym^d(\C X)\sminus\GS$, which admits a
continuous extension $F:D^2\to\Sym^d(\C X)$, the map 
$\widetilde f:\widetilde S^1\to\C X$ is bordant to zero in $\C X$.

By a small deformation make both $f$ and $F$ differentiable, 
make the image of $F$ intersecting only the highest dimensional  
stratum of the discriminant $\GS$ and make $F$ transversal to it. 

Now we apply a construction similar to the construction of section
\ref{s5.5}.  
Pull back the ramified
covering $\Sym^{d-1}(\C X)\times\C X\to\Sym^d(\C X)$ via $F$. This gives rise 
to the diagram
$$\begin{tikzcd}
\widetilde{D^2}\arrow[d,swap,"F^*q"]\arrow[r,"\widetilde F"]& \Sym^{d-1}(\C
X)\times\C X\arrow[d,"q"]\\
D^2\arrow[r,"F"]&\Sym^d(\C X)    
\end{tikzcd}
$$
The map $F^*q$, as a pull back of a ramified covering via a map transversal
to the ramification locus, is a 
covering ramified over the preimage of the ramification locus of
$q$, that is over a finite set $F^{-1}\GS$. Therefore, $\widetilde{D^2}$ is a
compact oriented surface with boundary.    
\end{proof}

\begin{rem}
Compare the proof of Theorem \ref{Th5} to the proof of Lemma 2.6 in 
Ozsv\'ath and Szab\'o \cite{OS}.  
Substantially, this is the same construction, but presented here more 
formally and explicitly. We could not limit ourselves to referring to
\cite{OS}, because we deal with a slightly more general situation: 
in \cite{OS} the number $d$ coincides with the genus of $\C X$. 
\end{rem}

\subsection{Dold-Thom theorem and tracing}\label{s5.7}
Let us point out a connection of the tracing map to the Dold-Thom
isomorphism.
 
Symmetric powers $\Sym^d(T)$ of a space $T$ with a base point $t_0\in T$ are 
embedded to each other 
$$\Sym^d(T)\to\Sym^{d+1}(T):\{x_1,\dots,x_d\}\mapsto\{x_1,\dots,x_d,t_0\}.$$ 
The union $\cup_d\Sym^d(T)$ is called {\em the infinite symmetric
power\/} of $T$ and denoted by $\Sym(T)$.

According to the Dold-Thom theorem \cite{DT}, for a reasonable space $T$ 
(say, a cw-complex), {\em there is a natural isomorphism 
$H_n(T)\to\pi_n(\Sym(T))$}. 

The tracing homomorphism $\TR$, which was introduced and used above only
for $T=\C X$, can be 
obtained in a more general situation as the composition 
$$\pi_1(\Sym^d(T))\xrightarrow{\inc_*}\pi_1(\Sym(T))\xrightarrow{DT^{-1}}H_1(T),$$ where
$\inc_*$ is the inclusion homomorphism and $DT^{-1}$ is the isomorphism
inverse to the isomorphism from the Dold-Thom theorem.

Unfortunately, this elegant construction is less convenient for our
purposes than the more explicit construction used above.

\section{Proofs of tracing theorems}\label{s6}
\subsection{Proof of Theorem \ref{Th3}}\label{s6.1}
Recall that Theorem \ref{Th3} claims that the tracing map  
$$\tr:\pi_1(\C\Leq(D)\sminus\C\GS\Leq(D))\to H_1(\C X)$$
is trivial for any simple divisor $D$ on a complex non-singular projective
curve $X$.

The diagrams \eqref{eq1} and \eqref{eq2} can be united 
as follows:
\begin{equation}\label{eq3}
\begin{tikzcd}
 \pi_1(\C\Leq(D)\sminus\C\GS\Leq(D))\arrow[rd,"\tr"]\arrow[r,"\inc_*"]&\pi_1(\UConf_d(\C
X))\arrow[d,"\Tr"]\arrow[r,"\inc_*"]&\pi_1(\Sym^d(\C X))\arrow[dl,"TR"]\\ 
&H_1(\C X)&
\end{tikzcd}
\end{equation}
On the top row of this diagram, we have two inclusion homomorphism. Replace
these two homomorphisms by their composition, which is also 
 an inclusion homomorphism:
\begin{equation}\label{eq4} 
\begin{tikzcd}
 \pi_1(\C\Leq(D)\sminus\C\GS\Leq(D))\arrow[rd,"\tr"]\arrow[rr,"\inc_*"]&
&\pi_1(\Sym^d(\C X))\arrow[dl,"TR"]\\ &H_1(\C X)&
\end{tikzcd}
\end{equation} 
Observe that the space $\C\Leq(D)\sminus\C\GS\Leq(D)$ of simple divisors 
linear equivalent
to $D$ is contained in   $\C\Leq(D)$ of all effective divisors
linear equivalent to $D$. Furthermore, $\C\Leq(D)\subset\Sym^d(\C X)$.  
Therefore we can insert $\C\Leq(D)$ into the top row of diagram \eqref{eq4}:
\begin{equation}\label{eq5} 
\begin{tikzcd}
 \pi_1(\C\Leq(D)\sminus\C\GS\Leq(D))\arrow[rd,"\tr"]\arrow[r,"\inc_*"]&
\pi_1(\C\Leq(D)))\arrow[r,"\inc_*"]&\pi_1(\Sym^d(\C X))\arrow[dl,"TR"]\\ 
&H_1(\C X)&
\end{tikzcd}
\end{equation} 
The space $\C\Leq(D)$ is known to be homeomorphic to a complex projective
space. Hence $\pi_1(\C\Leq(D))=0$. Therefore the
composition of the horizontal homomorphisms of diagram \eqref{eq5} is
trivial, as well as their composition
$\pi_1(\C\Leq(D)\sminus\C\GS\Leq(D))\xrightarrow{\inc_*}\pi_1(\Sym^d(\C X))$. Hence, 
the homomorphism 
$$\tr:\pi_1(\C\Leq(D)\sminus\C\GS\Leq(D))\to H_1(\C X)$$ is trivial. \qed  

\subsection{Proof of Theorem \ref{Th1}}\label{s6.2}
Let $D$ be a simple purely real divisor on a non-singular projective 
real algebraic curve $X$. Since $D$ is purely real, the imaginary tracing
$\Im\tr$ vanishes. 
Therefore the diagram \eqref{tr-diagram} reduces to 
\begin{equation}\label{purely-tr-diagram} 
\begin{tikzcd} 
\pi_1(\Def(D))\arrow[r,"\Re\tr"]\arrow[d,"\inc_*"]
& H_1(\R X)\arrow[d,"\inc_*"]\\
\pi_1(\C\Leq(D)\sminus\C\GS\Leq(D))\arrow[r,"\tr"]&H_1(\C X)
\end{tikzcd}
\end{equation} 
By Theorem \ref{Th3}, the homomorphism 
$\pi_1(\C\Leq(D)\sminus\C\GS\Leq(D))\xrightarrow{\tr}H_1(\C X)$
is trivial. Hence, the composition 
$\pi_1(\Def(D))\xrightarrow{\Re\tr} H_1(\R X)\xrightarrow{\inc_*}H_1(\C X)$
is trivial.

If $X$ is of type II (as in the part (c)), then the inclusion homomorpism 
$\inc_*: H_1(\R X)\to H_1(\C X)$ is injective (see section \ref{s2.2}).
Since $\inc_*\circ\Re\tr:\pi_1(\Def(D))\to H_1(\C X)$ is trivial and 
$\inc_*: H_1(\R X)\to H_1(\C X)$ is injective, commutativity of 
\eqref{purely-tr-diagram} implies that 
$\Re\tr:\pi_1(\Def(D))\to H_1(\R X)$ is trivial. 

If $X$ is of type I (as in the parts (a) and (b)), then the kernel of
$\inc_*: H_1(\R X)\to H_1(\C X)$ is an infinite cyclic group generated 
by the class $[\R X]$ defined by a complex orientation (see section 
\ref{s2.2}). Since the homomorphism
$$\inc_*\circ\Re\tr:\pi_1(\Def(D))\to H_1(\C X)$$ is 
trivial, the image of $\Re\tr$ must be contained in the kernel of 
 $\inc_*: H_1(\R X)\to H_1(\C X)$. This kernel is generated by 
$[\R X]=\sum_i[X_i]$,
where $X_i$ runs over all connected components of $\R X$.
This is what is claimed in the part (b). 

If $\R X$ has a connected component $X_j$ disjoint from $\supp D$
(as in part (a)), then the image of $\Re\tr$ does not fit to the kernel of
the inclusion homomorphism $\inc_*: H_1(\R X)\to H_1(\C X)$ unless
$\Re\tr=0$. \qed

\subsection{Proof of Theorem \ref{Th2}}\label{s6.3}
Since the curve $X$ is assumed to be of type I, $\C X\sminus\R X$ 
consists of two connected components, which are homeomorphic to each 
other by $\conj$. Denote the closures in $\C X$ of these components
by $X_+$ and $X_-$. 

Represent the real simple divisor $D$ as $\Re D+\Im D$, where
 $\Re D=\sum_{k=1}^nP_k$ and $\Im D=\sum_{k=1}^{m}(P_{n+k}+\conj P_{n+k})$
with $P_{n+k}\in X_+$.

Let $L:[0,1]\to \Def(D)$ be a loop with $L(0)=L(1)=D$. 
Then there are uniquely defined  
paths $L_1,\dots,L_n:I\to\R X$ and $L_{n+1},\dots,L_{n+m}:I\to X_+$ such that $L(t)=\sum_{k=1}^{n}L_k(t)+\sum_{k=1}^{m}(L_{n+k}(t)+\conj
L_{n+k}(t))$ for any $t\in I$. 

According to definition, the cycle $\sum_{k=1}^nL_k$ defines
$\Re\tr(L)\in H_1(\R X)$. Denote the homology class of the cycle
$\sum_{k=1}^mL_{n+k}$ in $X_+$ by $\Im_+\tr(L)$.

\begin{lem}The homology class  $\Im_+\tr(L)\in H_1(X_+)$ belongs to the
image of $H_1(\R X)$ under the inclusion homomorphism $H_1(\R X)\to
H_1(X_+)$.
\end{lem}

\begin{proof} 
The cycle 
$$\sum_{k=1}^nL_k+\sum_{k=1}^mL_{n+k}+\sum_{k=1}^m\conj L_{n+k}$$
realizes $\Tr(\Gl)\in H_1(\C X)$. 
 By Theorem \ref{Th3}, $\Tr(\Gl)=0$. 
The relativization homomorphism $H_1(\C X)\to H_1(\C X, X_-)$ 
maps $\Tr(\Gl)$ to the homology class realized by $\sum_{k=1}^mL_{n+k}$. 
It is zero, as the image of $0=\Tr(\Gl)$.
By excision, $H_1(\C X,X_-)=H_1(X_+,\R X)$. Hence, the cycle 
$\sum_{k=1}^mL_{n+k}$ realizes $0\in H_1(X_+,\R X)$. By exactness of the
homology sequence of pair $(X_+,\R X)$, it means that the homology class
realized by $\sum_{k=1}^mL_{n+k}$ in $X_+$ (that is $\Im_+\tr(\Gl)$) 
belongs to the image of inclusion homomorphism $H_1(\R X)\to
H_1(X_+)$.
\end{proof}

Let $\tau\in H_1(\R X)$ be a homology class such that its image under the
inclusion homomorphism $H_1(\R X)\to H_1(X_+)$ equals $\Im_+\tr(\Gl)$. 
Then the image of $\Re\tr(\Gl)+2\tau\in H_1(\R X)$ under the inclusion
homomorphism $H_1(\R X)\to H_1(\C X)$ is realized by the same cycle
$\sum_{k=1}^nL_k+\sum_{k=1}^mL_{n+k}+\sum_{k=1}^m\conj L_{n+k}$ as 
$\Tr(\Gl)$ which equals zero. 
Hence, $\Re\tr(\Gl)+2\tau\in H_1(\R X)$ belongs to the 
kernel of  the inclusion homomorphism $H_1(\R X)\to H_1(\C X)$. The kernel
is generated by a class of complex orientation $[\R X]$. 
Therefore, $\Re\tr(\Gl)+2\tau=c[\R X]$ for some $c\in\Z$. It follows
that $\Re\tr(\Gl)$ belongs to the subgroup generated by $[\R X]$ and even
elements. This is the statement (a) of Theorem \ref{Th2}.

If $\R X$ has a connected component disjoint from $\supp D$, then the
fundamental class of this connected component is not involved in
$\Re\tr(\Gl)$. Therefore in the formula $\Re\tr(\Gl)+2\tau=c[\R X]$
obtained above, $c$ must be even. Hence,  $\Re\tr(\Gl)=-2\tau+c[\R X]$ is
divisible by two. This is the statement (b) of Theorem 2. 
\qed

\section{Examples}\label{s7}


\subsection{Divisors of degree 2 on a plane M-cubic}\label{s7.1} 
 Let $X$ be a plane projective 
non-singular cubic curve with $\R X$ consisting of two components: an oval
$X_0$ and a one-sided component $X_1$. Let $A\in X_0$ and $B\in X_1$ be
arbitrary points. Consider divisor $D=A+B$. This is a simple purely real
divisor of degree 2. 

Assume that the points are chosen generically, so that 
the line $l$ passing through $A$ and $B$ is transversal to $X$. Then it
intersects $X$ in three points, two of them are $A$ and $B$. Denote the 
third intersection point by $C$. It belongs to $X_0$. For any other real
line $l'$ passing through $C$, the divisor $l'\cap\R X$ is of degree 3
and has a form $C+A'+B'$, where $A'\in X_0$ and $B'\in X_1$. For one of the
lines, namely for the tangent line to $X$ at $C$, $A'=C$. 
\begin{figure}[htb]
\begin{center}
\begin{tikzpicture}[scale=1.0,>=latex]
\filldraw(-.04,.18)circle(1.5pt);
\filldraw(1.17,.67)circle(1.5pt);
\node at(.15,.5){{\color{blue!70!black}$A$}};
\node at(.95,.85){{\color{blue!70!black}$B$}};
\draw[-](-1.5,-.4)--(2.0,1);
\draw[domain=-1:0,smooth=200,very thick,blue!70!black] plot({\x},{-((\x)^3-\x)^.5});
\draw[->,very thick,blue!70!black
](-0.02,-.05)--(-0.002,.07);
\draw[<-,very thick,blue!70!black](-.99,-.05)--(-.99,.07);
\draw[->,very thick,blue!70!black](1.01,-.05)--(1.01,.07);
\node at(-1.4,.2){{\color{blue!70!black}$X_0$}};
\node at(1.13,1.4){{\color{blue!70!black}$X_1$}};
\draw[domain=-1:0,smooth=200,very thick,blue!70!black] plot({\x},{((\x)^3-\x)^.5});
\draw[domain=1:1.4,smooth=200,very thick,blue!70!black] plot({\x},{-((\x)^3-\x)^.5});
\draw[domain=1:1.6,smooth=200,very thick,blue!70!black,->] plot({\x},{((\x)^3-\x)^.5});
\end{tikzpicture}
\hspace*{.7cm}
\begin{tikzpicture}[scale=1.0,>=latex]
\filldraw(-.98,-.19)circle(1.5pt);
\filldraw(-.04,.18)circle(1.5pt);
\filldraw(1.17,.67)circle(1.5pt);
\node at(-1.2,-.6){{\color{blue!70!black}$C$}};
\node at(.15,.5){{\color{blue!70!black}$A$}};
\node at(.95,.85){{\color{blue!70!black}$B$}};
\draw[-](-1.5,-.4)--(2.0,1);
\draw[domain=-1:0,smooth=200,very thick,blue!70!black] plot({\x},{-((\x)^3-\x)^.5});
\draw[->,very thick,blue!70!black](-0.02,-.05)--(-0.002,.07);
\draw[<-,very thick,blue!70!black](-.99,-.05)--(-.99,.07);
\draw[->,very thick,blue!70!black](1.01,-.05)--(1.01,.07);
\node at(-1.4,.2){{\color{blue!70!black}$X_0$}};
\node at(1.13,1.4){{\color{blue!70!black}$X_1$}};
\draw[domain=-1:0,smooth=200,very thick,blue!70!black] plot({\x},{((\x)^3-\x)^.5});
\draw[domain=1:1.4,smooth=200,very thick,blue!70!black] plot({\x},{-((\x)^3-\x)^.5});
\draw[domain=1:1.6,smooth=200,very thick,blue!70!black,->] plot({\x},{((\x)^3-\x)^.5});
\end{tikzpicture}
\hspace*{.7cm}
\begin{tikzpicture}[scale=1.0,>=latex]
\filldraw(-.98,-.19)circle(1.5pt);
\filldraw(-.04,.18)circle(1.5pt);
\filldraw(1.17,.67)circle(1.5pt);
\node at(-1.2,-.6){{\color{blue!70!black}$C$}};
\node at(.15,.5){{\color{blue!70!black}$A$}};
\node at(.95,.85){{\color{blue!70!black}$B$}};
\draw[-](-1.5,-.4)--(2.0,1);
\draw[-](-1.5,-.17)--(2.0,-.37);
\filldraw(-.05,-.24)circle(1.5pt);
\filldraw(1.05,-.31)circle(1.5pt);
\node at(.10,-.5){{\color{blue!70!black}$A'$}};
\node at(0.8,-0.53){{\color{blue!70!black}$B'$}};
\draw[domain=-1:0,smooth=200,very thick,blue!70!black] plot({\x},{-((\x)^3-\x)^.5});
\draw[->,very thick,blue!70!black](-0.005,+.01)--(-0.001,.06);
\draw[<-,very thick,blue!70!black](-.99,-.05)--(-.99,.07);
\draw[->,very thick,blue!70!black](1.01,-.05)--(1.01,.07);
\node at(-1.4,.2){{\color{blue!70!black}$X_0$}};
\node at(1.13,1.4){{\color{blue!70!black}$X_1$}};
\draw[domain=-1:0,smooth=200,very thick,blue!70!black] plot({\x},{((\x)^3-\x)^.5});
\draw[domain=1:1.4,smooth=200,very thick,blue!70!black] plot({\x},{-((\x)^3-\x)^.5});
\draw[domain=1:1.6,smooth=200,very thick,blue!70!black,->] plot({\x},{((\x)^3-\x)^.5});
\end{tikzpicture}
\end{center}
\end{figure}

The divisor $A'+B'$ is linear equivalent to $D$. Rotating $l$ about $C$
by $\pi$ gives rise to a loop in $\Def(D)$. The real tracing of this loop
equals $[\R X]$. The monodromy is trivial.

If on the same curve points $A$ and $B$ are selected either both on the
same component (both on $X_0$ or $X_1$ - does not matter), then the divisor 
$D=A+B$ has trivial tracing by Theorem \ref{Th1} (b).

\subsection{Divisors of degree 3 on a plane M-cubic}\label{s7.2} 
Let $X$, $A,B,C$ and $l$ be as in the preceding example. Let $D=A+B+C$. 
\begin{minipage}{\textwidth}
\begin{minipage}{4.5cm}
\begin{tikzpicture}[scale=1.0,>=latex]
\filldraw(-.5,0)circle(1pt);
\filldraw(-.98,-.19)circle(1.5pt);
\filldraw(-.04,.18)circle(1.5pt);
\filldraw(1.17,.67)circle(1.5pt);
\node at(-.3,-.1){$c$};
\node at(-1.2,-.6){{\color{blue!70!black}$C$}};
\node at(.15,.5){{\color{blue!70!black}$A$}};
\node at(.95,.85){{\color{blue!70!black}$B$}};
\draw[-](-1.5,-.4)--(2.0,1);
\draw[domain=-1:0,smooth=200,very thick,blue!70!black] plot({\x},{-((\x)^3-\x)^.5});
\draw[->,very thick,blue!70!black](-0.01,-.05)--(-0.01,.07);
\draw[<-,very thick,blue!70!black](-.99,-.05)--(-.99,.07);
\draw[->,very thick,blue!70!black](1.01,-.05)--(1.01,.07);
\node at(-1.4,.2){{\color{blue!70!black}$X_0$}};
\node at(1.13,1.4){{\color{blue!70!black}$X_1$}};
\draw[domain=-1:0,smooth=200,very thick,blue!70!black] plot({\x},{((\x)^3-\x)^.5});
\draw[domain=1:1.4,smooth=200,very thick,blue!70!black] plot({\x},{-((\x)^3-\x)^.5});
\draw[domain=1:1.6,smooth=200,very thick,blue!70!black,->] plot({\x},{((\x)^3-\x)^.5});
\end{tikzpicture}
\end{minipage}
\begin{minipage}{9.2cm}\vspace*{3pt} 
Choose a point $c\in l$ inside  the open disk
bounded by $X_0$. Then the divisors which are cut on $X$ by
real lines passing through $c$ form a subspace $S$ of $\Def(D)$ 
homeomorphic to circles. A generator of its fundamental group has
tracing $[\R X]$. The monodromy is the transposition of $A$ and
$B$.%
\end{minipage}
\end{minipage}

\subsection{Divisors of degree 9 on a plane M-cubic}\label{s7.3}
One of the first non-trivial special cases of the problems under 
consideration appears, when the curve is a plane projective M-cubic and 
the divisor is cut on it by another cubic curve. This 
special case was considered by Ay\c{s}eg\"ul \"Ozt\"urkalan \cite{Ay}.
 
Let $X$, $X_0$, $X_1$ and $c$ be as in section \ref{s7.2}. 
Each real line passing through $c$ intersects
$X_0$ in two points and $X_1$ in one point. This defines a two-fold
covering $X_0\to X_1$. Choose on $X_0$ and $X_1$ orientations such
that with respect to them the covering is of degree $+2$. (These orientations 
are defined up to simultaneous reversing, they are the complex
orientations of $\R X$.)

Let $A$ be a set that is cut on $\R X$  
by other real cubic curve $Y$. Denote $X_i\cap A$ by $A_i$.
By the Bezout Theorem, $A$ may consist of at most 9 points. 
Assume that $A$ {\em consists of exactly 9 points.}
One can easily see that the number of points in $A_0$ is even, and number 
of points and $A_1$ is odd.

 Points of $A_i$ are cyclically ordered by their 
position on $X_i$ according to the orientation. 
\def\ah{Latex[length=1.9mm,width=1.6mm]}

\begin{figure}[htb]
\begin{center}
\begin{tikzpicture}[scale=.9,>=latex]
\filldraw(-.5,0)circle(1pt);
\draw[-](-1.5,-.4)--(2.0,1);
\draw[domain=-1:0,smooth=200,very thick,blue!70!black] plot({\x},{-((\x)^3-\x)^.5});
\draw[->,very thick,blue!70!black](-0.01,-.05)--(-0.01,.07);
\draw[<-,very thick,blue!70!black](-.99,-.05)--(-.99,.07);
\draw[->,very thick,blue!70!black](1.01,-.05)--(1.01,.07);
\node at(-1.4,.2){{\color{blue!70!black}$X_0$}};
\node at(1.05,1.3){{\color{blue!70!black}$X_1$}};
\node at(-.3,-.1){$c$};
\draw[domain=-1:0,smooth=200,very thick,blue!70!black] plot({\x},{((\x)^3-\x)^.5});
\draw[domain=1:1.4,smooth=200,very thick,blue!70!black] plot({\x},{-((\x)^3-\x)^.5});
\draw[domain=1:1.6,smooth=200,very thick,blue!70!black,->] plot({\x},{((\x)^3-\x)^.5});
\end{tikzpicture}
\hspace*{1cm}
\begin{tikzpicture}[scale=.9,>=latex]
\draw[domain=-1:0,smooth=200,very thick,blue!70!black] plot({\x},{-((\x)^3-\x)^.5});
\node at(-1.3,-.2){{\color{blue!70!black}$X_0$}};
\node at(1.05,1.3){{\color{blue!70!black}$X_1$}};
\node at(2.3,.8){{\color{green!30!black}$\R Y$}};
\draw[domain=-1:0,smooth=200,very thick,blue!70!black] plot({\x},{((\x)^3-\x)^.5});
\draw[domain=1:1.4,smooth=200,very thick,blue!70!black] plot({\x},{-((\x)^3-\x)^.5});
\draw[domain=1:1.6,smooth=200,very thick,blue!70!black] plot({\x},{((\x)^3-\x)^.5});
\draw[domain=-.58:.61,smooth=200,thick,green!50!black] plot({38*(\x+.5)*(\x-.5)*\x},{\x});
\filldraw(1.12,.55)circle(1pt);
\filldraw(1.08,-.425)circle(1pt);
\filldraw(1.01,-.12)circle(1pt);
\filldraw(0,0)circle(1pt);
\filldraw(-.247,.482)circle(1pt);
\filldraw(-.875,.444)circle(1pt);
\filldraw(-.985,.111)circle(1pt);
\filldraw(-.81,-.534)circle(1pt);
\filldraw(-.285,-.515)circle(1pt);
\end{tikzpicture}
\hspace*{1cm}
\begin{tikzpicture}[scale=.9,>=latex]
\draw[domain=-1:0,smooth=200,very thick,blue!70!black] plot({\x},{-((\x)^3-\x)^.5});
\node at(-.5,-.1){{\color{blue!70!black}$X_0$}};
\node at(1.05,1.3){{\color{blue!70!black}$X_1$}};
\draw[domain=-1:0,smooth=200,very thick,blue!70!black] plot({\x},{((\x)^3-\x)^.5});
\draw[domain=1:1.4,smooth=200,very thick,blue!70!black] plot({\x},{-((\x)^3-\x)^.5});
\draw[domain=1:1.6,smooth=200,very thick,blue!70!black] plot({\x},{((\x)^3-\x)^.5});
\filldraw(1.12,.55)circle(1pt);
\draw[-\ah](1.01,-.12) to [out=130,in=-90] (.82,.23) 
to [out=90,in=210] (1.12,.55);
\filldraw(1.01,-.12)circle(1pt);
\draw[-\ah](1.08,-.425) to [out=180,in=-90] (.88,-.31) 
to [out=90,in=220] (1.01,-.12);
\filldraw(1.08,-.425)circle(1pt);
\draw[-\ah](1.12,.55) to [out=-30,in=90] (1.5,0.05) 
to [out=-90,in=30] (1.08,-.425);
\filldraw(0,0)circle(1pt);
\draw[-\ah](0,0) to [out=40,in=-90] (.15,.3) 
to [out=90,in=10] (-.247,.482);
\filldraw(-.247,.482)circle(1pt);
\draw[-\ah](-.247,.482) to [out=110,in=0] (-.58,.8) 
to [out=180,in=100](-.875,.444) ;
\filldraw(-.875,.444)circle(1pt);
\draw[-\ah](-.875,.444) to [out=170,in=70] (-1.17,.34) 
to [out=260,in=160] (-.985,.111);
\filldraw(-.985,.111)circle(1pt);
\draw[-\ah](-.985,.111) to [out=200,in=100] (-1.22,-.34) 
to [out=280,in=190] (-.81,-.534);
\filldraw(-.81,-.534)circle(1pt);
\draw[-\ah](-.81,-.534) to [out=-90,in=180] (-.55,-.8) 
to [out=0,in=-90] (-.285,-.515);
\filldraw(-.285,-.515)circle(1pt);
\draw[-\ah](-.285,-.515) to [out=-30,in=-110] (.16,-.34) 
to [out=70,in=-45] (0,0);
\end{tikzpicture}

\end{center}
\label{f1}
\end{figure}

A continuous change of $Y$ forces $A$ to move along $\R X$. We consider only 
those changes of $Y$ during which $\R Y$ stays transversal to $\R X$. Assume 
that $A$ moved for a while and came back to its original position. This 
gives rise to a monodromy permutation $A\to A$, which splits to permutations 
$A_0\to A_0$ and $A_1\to A_1$. The cyclic order in $A_i$ defined by the 
complex orientation of $X_i$ does not change. So, the permutations 
$A_i\to A_i$ are cyclic.   Ay\c{s}eg\"ul \"Ozt\"urkalan \cite{Ay} proved 
that 

{\it If none of
$A_i$ is empty, then the permutations on $A_0$ and $A_1$ are closely 
related to each other:
the pair of permutations are realized if and only if they are the same 
powers of the cyclic 
permutations moving each point of $A_i$ to the next one along the complex
orientation of $X$. 

If $A_0=\varnothing$, then only the identity permutation 
of $A_1$ is realizable.} 

These restrictions on the permutations follow from Theorem
\ref{Th1} (a) and (b) of this paper and the relation between the monodromy 
and tracing described in section \ref{s3.3} above.  

\subsection{Separating morphisms}\label{s7.4}
In the example of section \ref{s7.2} the projection $X\to P^1$ from 
point $c$ is such
that the preimage of $\R P^1$ is $\R X$. Morphisms of real algebraic curves 
with this property are called {\em separating.\/} The name was motivated by 
the fact that existence of a separating morphism $X\to P^1$ implies that $X$
is of type I, i.e., $\R X$ separates halves of $\C X$ from each other.

Originally separating morphisms were used to prove that the curve under
consideration is of type I. For example, if $X$ is a real plane projective 
curve of degree $d$, which has a nest of depth $[\frac{d}2]$, then
projection from any real point which surrounded by ovals of the nest is a
separating morphism. Hence $X$ is of type I. 

Gabard \cite{Gab06} proved that existence of a separating morphism is
necessary and sufficient condition for a curve being of type I. 

Notice that if $f:X\to P^1$ is a separating morphism, then the family 
$\{f^*(y)\mid y\in\R P^1\}$ consists of
purely real simple divisors. For any $D\in\{f^*(y)\mid y\in\R P^1\}$, the
image of $\Re\tr:\pi_1(\Def(D))\to H_1(\R X)$ is generated by the class 
of a complex orientation of $X$.  

\begin{rem}\label{rem4.2.1} 
In all examples that I am aware about, if $D$ is a simple purely real
divisor on a real non-singular projective curve $X$ such that the 
$\Re\tr:\pi_1(\Def(D))\to H_1(\R X)$ is not trivial, there exists  
a separating morphism $f:X\to P^1$ such that $D=f^*(y)$ for some 
$y\in\R P^1$. Is this always true?
\end{rem}

\subsection{Looking at the dual curve}\label{s7.5}
In the example of section \ref{s7.2}, the divisors linear equivalent to 
$D$ are cut on $X$ by lines. Lines form the projective plane
$\check P^2$ dual to $P^2$. 
Thus there is a bijection between $\Leq(D)$ and 
$\check P^2$. 

 Non-simple divisors in $\Leq(D)$ correspond to 
points of the curve $\check X$ dual to $X$. Simple divisor in $\Leq(D)$
correspond to points of its complement $\C\check P^2\sminus\C\check X$. 
The set $\Def(D)$ is a connected
component of the complement of  $\R\check X$ in the 
$\R\Leq(D)$.  Let us look at the curve $\R\check X$.

\begin{center}
\includegraphics[scale=.23]{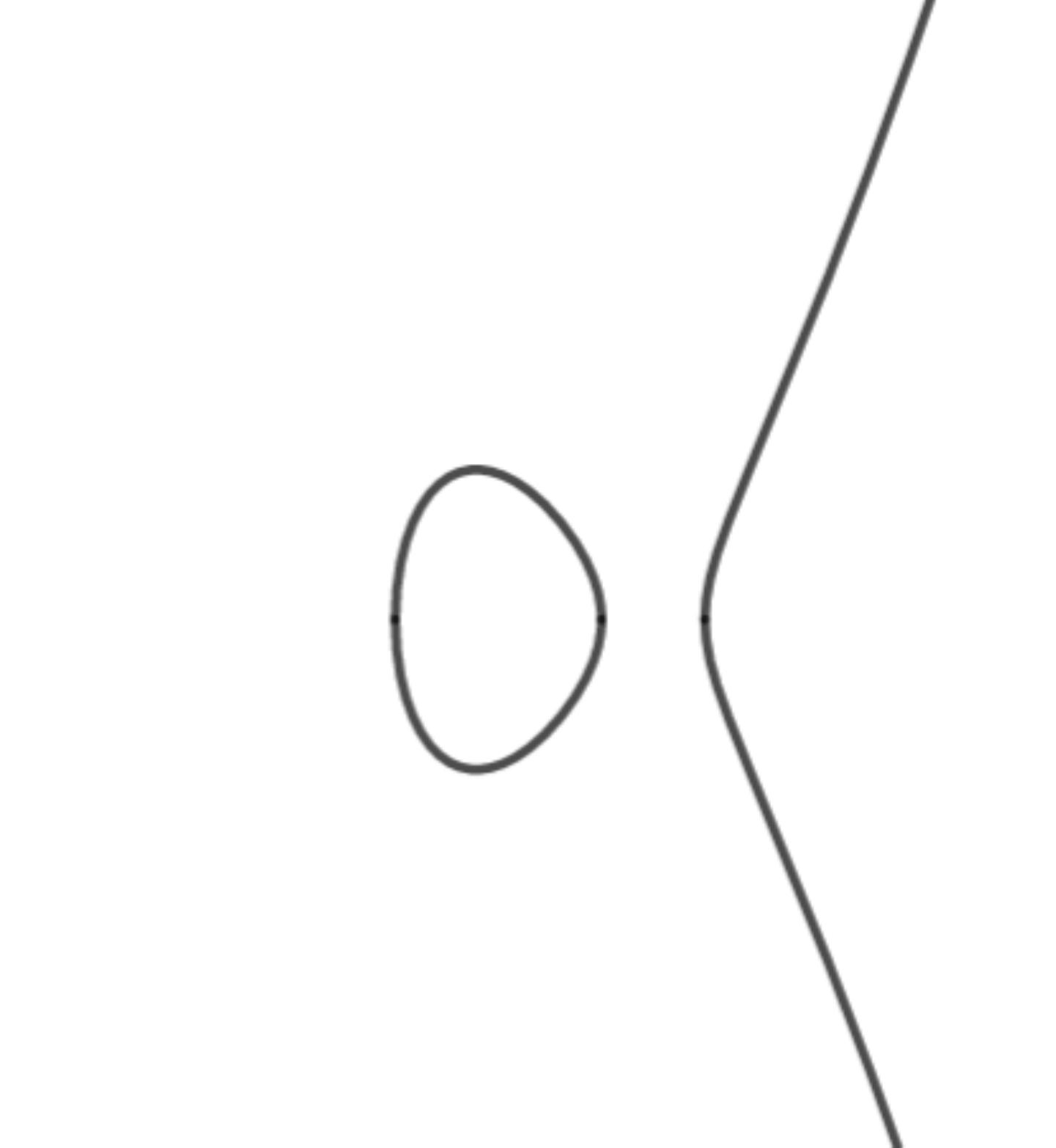}\hspace*{2cm}
\includegraphics[scale=.2]{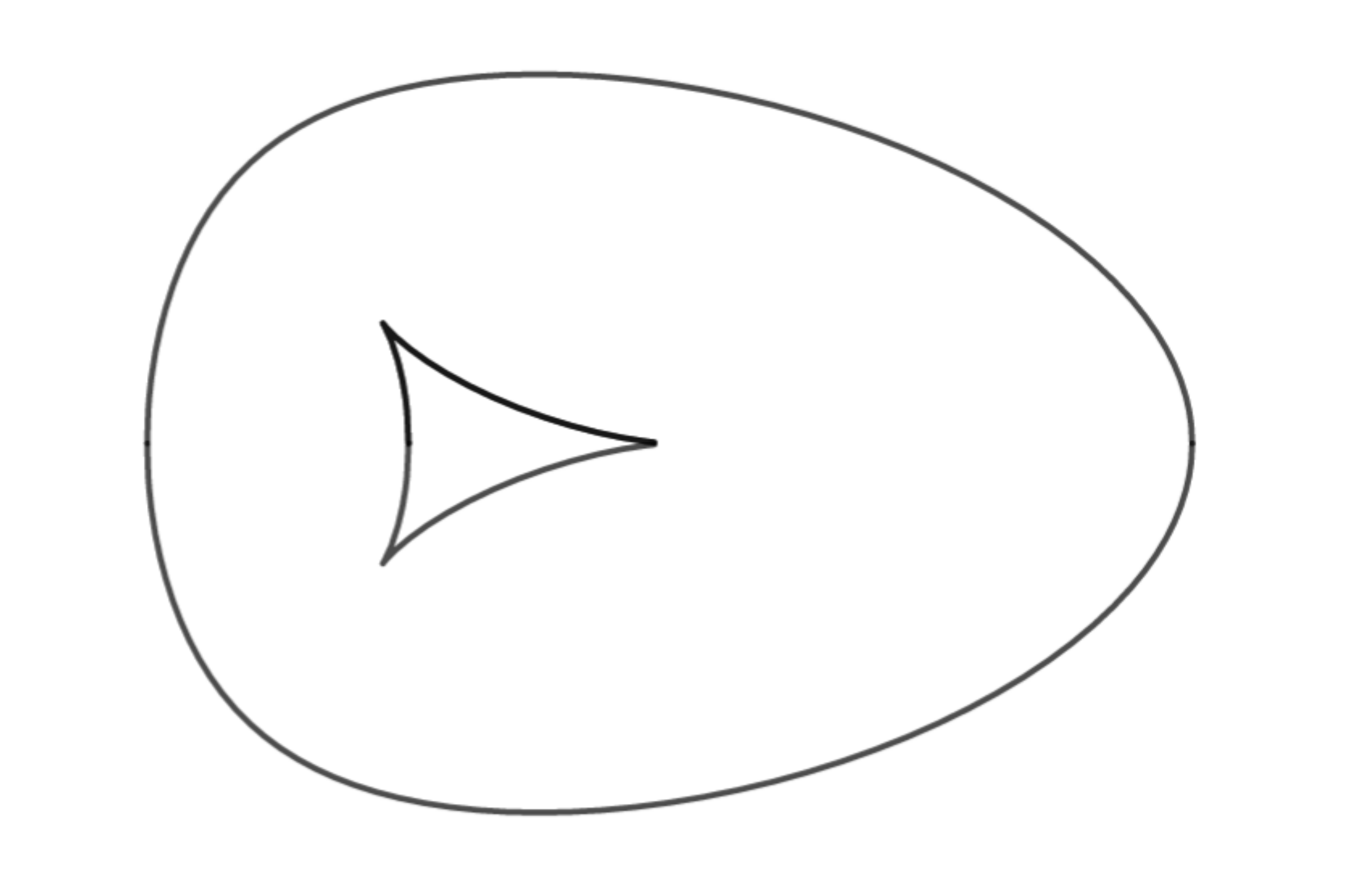}\\
M-cubic \hspace*{4cm} the dual curve
\end{center}
$\R\check X$ consists of two connected components. The component $\check
X_1$ dual to the one-sided component $X_1$ has three cusps (which
correspond to inflection points on $X_1$). The component $\check X_0$ dual
to the oval $X_0$ is an oval enclosing $\check X_1$. 

Each connected component of  $\R\check P^2\sminus\R\check X$ is $\Def(E)$
for a divisor $E$ corresponding to any of its points. 
The non-orientable component of $\R\check P^2\sminus\R\check X$ is $\Def(D)$. It contains 
lines (in particular, the infinity-line) which correspond to pencils of line 
passing through points encircled by $X_0$. The fundamental group of this
component is $\Z$. Points of this connected
component of $\R\check P^2\sminus\R\check X$ correspond to 
purely real divisors each of which consists of two points on $X_0$ and a
point on $X_1$. 

The component enclosed in $\check X_1$ consists
of purely real divisors, which are collinear triples of points on $X_1$.
The fundamental group of this component is trivial.

The third component is bounded by both  $\check X_0$ and $\check X_1$.
It consists of divisors with a single real point and two complex conjugate
points. The fundamental group of this component is $\Z$. The image of 
real tracing homomorphism is generated by $2[X_1]$. 

This example illustrates part (a) of Theorem \ref{Th2}. Indeed, in this
example a real divisor on an
M-cubic curve cut on it by a real line intersecting the one-sided component 
in a single point and missing the oval, and the image of $\Re\tr$ is 
generated by $2[X_1]$.  

\subsection{A generalization of dual curve}\label{s7.6} Let $X$ be a
non-singular real projective curve and $D$ a simple real divisor on $X$. 
The variety $\Leq(D)$ of divisors linear equivalent to $D$ is a real 
projective space. Its complex points are divisors linear equivalent to $D$. 
Non-simple divisors form a real algebraic hypersurface of $\Leq(D)$. 
In section \ref{s7.5} this hypersurface was identified with curve 
$\check X$ dual to $X$.

In general, when $\dim \Leq(D)>2$, we still have a real projective space
$\R\Leq(D)$ and a real projective variety $\GS\Leq(D)$ of non-simple
divisors linear equivalent to $D$. The complement
$\R\Leq(D)\sminus\R\GS\Leq(D)$ is an open set, its connected components 
are the spaces of deformation $\Def(E)$ for real simple divisors $E$
linear equivalent to $D$. Of course, a knowledge on geometry of $\GS\Leq(D)$
may be quite valuable for understanding of $\Def(E)$.

\end{document}